\documentclass[12pt]{article}
\usepackage{latexsym, amsmath, amsfonts, amsthm, amssymb}
\usepackage{times}
\usepackage{tikz-cd}
\usepackage{a4wide}
\usepackage{times}
\usepackage{graphicx, tocloft}
\usepackage{mathrsfs}
\usepackage{xcolor}
\usepackage{hyperref}

\def \Vol {\mathrm{Vol}}
\def \Area {\mathrm{Area}}

\def \Ric {\text{\rm Ric}}

\def \RR {\mathbb R}

\def \eps {\varepsilon}
\def \vphi {\varphi}

\def \cM {\mathcal M}

\def \cL {\mathcal L}

\def \cF {\mathcal F}

\def \cU {\mathcal U}
\def \cL {\mathcal L}

\def \arccot {{ \rm arccot}}

\newtheorem{theorem}{Theorem}[section]

\newtheorem*{theorem*}{theorem}
\newtheorem{definition}[theorem]{Definition}

\newtheorem{lemma}[theorem]{Lemma}

\newtheorem{proposition}[theorem]{Proposition}
\newtheorem*{proposition*}{Proposition}
\newtheorem{corollary}[theorem]{Corollary}
\newtheorem{example}[theorem]{Example}

\theoremstyle{definition}
\newtheorem{remark}[theorem]{Remark}

\def\myffrac#1#2 in #3{\raise 2.6pt\hbox{$#3 #1$}\mkern-1.5mu\raise 0.8pt\hbox{$
		#3/$}\mkern-1.1mu\lower 1.5pt\hbox{$#3 #2$}}

\def\qed{\hfill $\vcenter{\hrule height .3mm
		\hbox {\vrule width .3mm height 2.1mm \kern 2mm \vrule width .3mm
			height 2.1mm} \hrule height .3mm}$ \bigskip}

\def \susbeteq {\subseteq}

\def \ola {\overleftarrow}

\title{Brunn-Minkowski inequalities for sprays on surfaces}
\author{Rotem Assouline}
\date{}
\AtEndDocument{%
  \par
  \medskip
  \begin{tabular}{@{}l@{}}%
    \textsc{Department of Mathematics, Weizmann Institute of Science, Rehovot 76100 Israel}\\
    \textit{E-mail address}: \texttt{rotem.assouline@weizmann.ac.il}
  \end{tabular}}

\begin{document}

\maketitle

\begin{abstract}
	We propose a generalization of the Minkowski average of two subsets of a Riemannian manifold, in which geodesics are replaced by an arbitrary family of parametrized curves. Under certain assumptions, we characterize families of curves on a Riemannian surface for which a Brunn-Minkowski inequality holds with respect to a given volume form. In particular, we prove that under these assumptions, a family of constant-speed curves on a Riemannian surface satisfies the Brunn-Minkowski inequality with respect to the Riemannian area form if and only if the geodesic curvature of its members is determined by a function $\kappa$ on the surface, and $\kappa$ satisfies the inequality $$K + \kappa^2 - |\nabla\kappa| \ge 0$$ where $K$ is the Gauss curvature.
\end{abstract}

\section{Introduction}

The Brunn-Minkowski inequality asserts that for every $A,B\subseteq \RR^n$, Borel measurable and nonempty, and for every $0 < \lambda < 1$,
\begin{equation}\label{BM} \Vol_n( (1-\lambda) A + \lambda B )^{1/n} \geq (1-\lambda) \cdot \Vol_n(A)^{1/n} + \lambda \cdot \Vol_n(B)^{1/n}. 
\end{equation}
Here, $\Vol_n$ denotes the Lebesgue measure, and $$(1-\lambda) A + \lambda B = \{(1-\lambda)a + \lambda b \mid a \in A , b \in B\}.$$ For background on the classical Brunn-Minkowski inequality, see Schneider \cite[Section 7.1]{schneider}. The Brunn-Minkowski inequality admits a generalization to Riemannian manifolds, first proved in its stronger, functional version (the Borell-Brascamp-Lieb inequality) by Cordero-Erausquin, McCann and Schmuckenschl\"ager \cite{CMS}. It states that if $M$ is a complete, $n$-dimensional Riemannian manifold with nonnegative Ricci curvature, then \eqref{BM} holds true, with $\Vol_n$ replaced by Riemannian volume, and with the set $(1-\lambda)A + \lambda B$ naturally replaced by the set of points of the form $\gamma(\lambda)$, where $\gamma:[0,1] \to M$ is a constant-speed minimizing geodesic joining the set $A$ to the set $B$ (for example, if $\lambda = 1/2$ then this equals the set of midpoints of geodesic segments joining the two sets). The Riemannian Brunn-Minkowski inequality in this formulation first appeared in Sturm \cite{sturm}. The validity of the Brunn-Minkowski inequality for all $A,B$ and $\lambda$ is in fact equivalent to nonnegative Ricci curvature, see \cite{MPR}.

\medskip
The Brunn-Minkowski inequality combines two elements: volume and Minkowski summation. There is an extensive body of research in which the Riemannian volume form is replaced by an arbitrary measure $\mu$ with a smooth density, see \cite{BGL,LV}. The Brunn-Minkowski inequality, with the exponent $1/n$ replaced by $1/N$ for some $N\in(-\infty,0]\cup(n,\infty]$, then follows from nonnegativity of the generalized Ricci tensor $\Ric_{g,\mu,N}$, see \cite{sturm,Oh16,Oh}.

\medskip
In the present paper, we focus on the second element of the Brunn-Minkowski inequality: Minkowski summation. Instead of interpolating between two sets using geodesics, we propose to use an arbitrary family of curves. Restricting to the two-dimensional case, we seek to characterize families of curves for which a Brunn-Minkowski inequality analogous to \eqref{BM} holds.

\medskip
Let $(M,g)$ be a two-dimensional oriented Riemannian manifold. Fix a family $\Gamma$ of smooth unit-speed curves on $M$, with the property that for any unit vector $v \in SM$, there is a unique curve $\gamma \in \Gamma$ with $\dot\gamma(0) = v$. Here $SM$ is the unit tangent bundle of $M$. Equivalently, fix a function ${k} : SM \to \RR$ and let $\Gamma$ be the family of solutions to the ordinary differential equation
	$$\nabla_{\dot\gamma}\dot\gamma = {k}(\dot\gamma)\dot\gamma^\perp,$$
where $^\perp$ denotes rotation by $\pi/2$ in the positive direction. Thus ${k}$ prescribes the geodesic curvature of a curve $\gamma \in \Gamma$. Include in the family also all constant-speed orientation-preserving reparametrizations of the unit-speed curves above. Such a family of curves is uniquely determined by a vector field on $TM$, called a \emph{spray}, which we shall also denote by $\Gamma$. The curves are then called $\Gamma$-\emph{geodesics}. The precise definition of a spray, as well as definitions of the technical assumptions on $\Gamma$ appearing in the formulation of Theorem \ref{mainthm} below, are given in Section \ref{prelsec}.

\medskip
For two subsets $A,B \subseteq M$, denote by $\cM_\Gamma(A,B;\lambda)$ the set of points of the form $\gamma(\lambda)$, where $\gamma$ is a $\Gamma$-geodesic satisfying $\gamma(0) \in A$ and $\gamma(1) \in B$. This is a generalization of the Minkowski average  $(1-\lambda)A + \lambda B$ of two sets $A,B \subseteq\RR^2$, since we can take our family to be all constant-speed lines on the Euclidean plane. We prove the following:	 

\begin{theorem}\label{mainthm}
	Let $\Gamma$ be a simple, proper metric spray on $M$. The following are equivalent:
	\begin{enumerate}
		\item The function ${k}$ is independent of the direction, i.e. there exists a smooth function $\kappa: M \to \RR$ such that ${k} = \kappa\circ\pi$ where $\pi: SM \to M$ is the bundle projection, and moreover 
		\begin{equation}\label{mainthmeq}K + \kappa^2 - |\nabla\kappa| \ge 0,\end{equation}
		where $K : M \to \RR$ is the Gauss curvature of $g$.
		\item For every pair of Borel, nonempty subsets $A,B \subseteq M$ and every $0 < \lambda < 1$,
		$$\Area(\cM_\Gamma(A,B;\lambda))^{1/2} \ge (1-\lambda)\cdot \Area(A)^{1/2} + \lambda \cdot \Area(B)^{1/2},$$
		where $\Area$ denotes Riemannian area.
	\end{enumerate}
\end{theorem}
Theorem \ref{mainthm} is a special case of Corollary \ref{simplemetricBMcor}, in which the Riemannian area form is replaced by an arbitrary volume form on $M$.

\medskip
Theorem \ref{mainthm} generalizes the horocyclic Brunn-Minkowski inequality in the hyperbolic plane which was proved in \cite{AK}, as well as the classical Brunn-Minkowski in $\RR^2$. Moreover, by setting $\kappa \equiv 0$, we recover the equivalence between the (geodesic) Brunn-Minkowski inequality and the nonnegativity of the Gauss curvature (at least in the case where the geodesic spray is simple and proper). More examples are given in Section \ref{examplesec}. 

\medskip
The main tool in the proof of the (harder) direction $1 \implies 2$ in Theorem \ref{mainthm} is the needle decomposition technique, suggested and developed in the Riemannian setting by Klartag \cite{K1}. For our purposes we use a needle decomposition theorem for geodesically-convex Finsler surfaces whose proof can be found in \cite{AK} and is based on the proofs found in Klartag \cite{K1}, Ohta \cite{Oh} and Cavalletti and Mondino \cite{CM1}. The idea behind the needle decomposition technique is to decompose the space into $\Gamma$-geodesic arcs, localize the desired inequality into a one-dimensional inequality on each arc, and integrate the localized inequality. The technique employs $L^1$-mass transport, and requires the $\Gamma$-geodesics to be minimizing geodesics of some Finsler metric, possibly after reparametrization. Fortunately, simple, proper metric sprays satisfying the Brunn-Minkowski inequality with respect to some volume form do have this property. Simplicity and properness are used mainly in this step; they can probably be replaced by other, possibly weaker, topological assumptions on the spray.

\medskip
The fact that $\dim M = 2$ is used in several places. Firstly, our local analysis of the spray and derivation of condition \eqref{mainthmeq} is two-dimensional, but we believe that such a computation can be carried out in higher dimensions, at least for some classes of sprays. In a forthcoming paper we analyze magnetic sprays on manifolds of arbitrary dimensions, i.e. solutions to ODEs of the form $\nabla_{\dot\gamma}\dot\gamma = |\dot\gamma|\mathrm{Y}\dot\gamma$, where $\mathrm{Y}$ is linear. Secondly, and most significantly, we rely on projective Finsler-metrizability of the spray. In dimension two, Brunn-Minkowski implies that the spray is magnetic, which in turn enables us to Finsler-metrize it. We do not know what the situation is in this regard in higher dimensions. Of course, there might be an approach to proving the Brunn-Minkowski inequality for sprays which does not require projective Finsler-metrizability. Finally, the needle decomposition theorem \ref{needledecomp} was proved in \cite{AK} in dimension two. But most of the steps in the proof presented in \cite{AK} were already carried out in \cite{K1,Oh} in general dimension, so we believe that it doesn't take much to extend it to arbitrary dimension.

\medskip
The paper is organized as follows. In Section \ref{prelsec} we first recall some basic facts about Riemannian geometry of surfaces and about sprays. We then prove Proposition \ref{metrizabilityppn} regarding projective Finsler-metrizability of magnetic sprays. Finally, we give the necessary background and relevant results on needle decomposition. In Section \ref{weightedsec} we introduce the notion of a {\it nonnegatively curved weighted spray space}, and give a characterization of such spaces in the case of a metric spray on a Riemannian surface. We also mention an analogue of the curvature-dimension condition from the theory of metric measure spaces \cite{sturm,LV} in the setting of sprays on surfaces. In Section \ref{BMsec} we prove the equivalence between the nonnegative curvature condition and the Brunn-Minkowski inequality in the case of simple, proper metric sprays (and, more generally, in the case of projectively Finsler-metrizable sprays on surfaces). In Section \ref{examplesec} we provide some examples of weighted spray spaces satisfying the Brunn-Minkowski inequality.

\medskip
{\it Acknowledgements.}  I would like to express my deep gratitude to my Ph.D. advisor, prof. Bo'az Klartag, for his close guidance and support. I would also like to thank the referee for a thorough reading of the manuscript and for useful comments.

\medskip
The work is part of the author's Ph.D. research conducted at the department of mathematics at the Weizmann Institute, Rehovot, Israel. Some of the work was done during a semester program held at ICERM, Brown University, Providence, RI. Supported by a grant from the Israel Science Foundation (ISF).

\section{Preliminaries}\label{prelsec}

By a \emph{Riemannian surface} we mean a two-dimensional oriented Riemannian manifold. For a tangent vector $v \in TM$ we denote by $v^\perp$ the vector $v$ rotated by $\pi/2$ in the positive direction, so that $v$ and $v^\perp$ have the same norm and $(v,v^\perp)$ is an oriented orthogonal basis of the tangent space. We will occasionally use the Hodge star $\star$, which for one-forms on a Riemannian surface is simply the operator 
$$\star \eta = ((\eta^\sharp)^\perp)^\flat,$$
where $\sharp$ and $\flat$ are the musical isomorphisms. 

\medskip
We shall use the letter $\pi$ to denote both the bundle projection from $TM$ to $M$ and the bundle projection from $TTM$ to $TM$. Thus from $TTM$ to $TM$ we have two canonical maps: the bundle projection $\pi$, and the differential $d\pi$ of the bundle projection $\pi: TM\to M$. 

\medskip
We recall some basic facts about Riemannian surfaces which can be found in \cite{ST}.  The unit tangent bundle $SM$ of a Riemannian surface is the subbundle of $TM$ consisting of unit vectors,
$$SM : = \{v \in TM \mid |v|_g = 1\}.$$

\medskip
The unit tangent bundle admits a natural global frame $(E_1,E_2,V)$. The flow of the vector field $E_1$ is the geodesic flow on $SM$,  the flow of $V$ restricts to rotation of each tangent circle (chosen according to the orientation of $M$), and we have the commutation relations
\begin{equation}\label{commuterel}[V,E_1] = E_2, \qquad [V,E_2] = - E_1, \qquad [E_1,E_2] = KV.\end{equation}
The dual coframe $(\theta_1,\theta_2,\psi)$ satisfies the relations
\begin{equation}\label{theta12def} d\pi(\xi) = \theta_1(\xi)\pi(\xi) + \theta_2(\xi)\pi(\xi)^\perp, \qquad \xi \in TSM,\end{equation}
and the structure equations
\begin{align}\label{structure}
	d\theta_1 = \psi\wedge\theta_2, \qquad d\theta_2 = -\psi\wedge\theta_1, \qquad d\psi = - K \, \theta_1 \wedge \theta_2.
\end{align}
Here $K$ is the Gauss curvature of the surface, which we view as a function on $SM$ which is constant on each fiber. The Riemannian volume form, denoted by $\omega_g$, satisfies
\begin{align}\label{omegagpullback}
	\pi^*\omega_g = \theta_1\wedge\theta_2.
\end{align}

It is useful for us to consider the extensions of these objects to the full tangent bundle $TM$. On $TM$ we have the canonical radial vector field $R$ (sometimes called the Liouville vector field), which is the infinitesimal generator of the flow $v \mapsto e^t v$. Alternatively, for $v \in T_xM$, if by abuse of notation we identify $T_xM$ with a subspace of $T_vTM$ via translation, then $$R\vert_v = v.$$
We again denote by $E_1$ the infinitesimal generator of the geodesic flow on $TM$, and by $V$ the infinitesimal generator of rotation in the positive direction on each fiber of $TM$. Using \eqref{commuterel}, \eqref{structure} and the homogeneity of $E_1$ and $V$, one can easily prove that on the full tangent bundle we have the commutation relations 
\begin{align*}
	[V,E_1] & = E_2, & [V,E_2] & = - E_1, & [E_1,E_2] & = KV,\\
	[R,E_1] & = E_1, & [R,E_2] & = E_2, & [R,V] & = 0,
\end{align*} 
and the dual coframe $(\theta_1,\theta_2,\psi,\rho)$ satisfies \eqref{theta12def} as well as the structure equations
\begin{align}\label{structureTM}
\begin{split}
	d\theta_1 & = \theta_1\wedge\rho + \psi\wedge\theta_2,\\
	d\theta_2 & = \theta_2\wedge\rho - \psi\wedge\theta_1,\\
	d\psi & = - K  \, \theta_1\wedge\theta_2,\\
	d\rho & = 0.
\end{split}
\end{align}
Here (and from now on) we extend the Gauss curvature $K$ to a 2-homogeneous function on $TM$ whose value on $SM$ coincides with the usual Gauss curvature. 

\medskip
Since the metric $g$, viewed as a function on $TM$, $v \mapsto g(v,v)$, is fiberwise 2-homogeneous and is invariant under the flows of $E_1,E_2$ and $V$, we have
$$dg = 2g\rho.$$

The following lemma is proved easily using the formula relating the Christoffel symbols of conformal metrics. Observe that two conformal metrics have the same vector field $V$.
\begin{lemma}\label{conformalspraylemma}
	Let $(M,g)$ be a Riemannian surface and let $\tilde g = e^{2f}g$ be a metric conformal to $g$. Then the geodesic sprays of $g,\tilde g$ are related by
	$$\tilde E_1 = E_1 - \star df \, V - df \, R.$$
\end{lemma}

\subsection{Sprays}\label{spraysec}
A \emph{spray} on a manifold $M$ is a vector field $\Gamma : TM \to TTM$ satisfying 
\begin{enumerate}
	\item The semispray condition: $d\pi \circ \Gamma = \mathrm{id}$, where $\mathrm{id}$ is the identity on $TM$, and
	\item Homogeneity: $[R,\Gamma] = \Gamma$.
\end{enumerate}

The semispray condition means that the integral curves of $\Gamma$ are canonical lifts of curves on $M$, i.e. if $\tilde \gamma$ is an integral curve of $\Gamma$ and $\gamma : = \pi\circ\tilde\gamma$ is the projection of $\tilde \gamma$ to $M$, then $\dot\gamma = \tilde\gamma$. The homogeneity condition means that if $\tilde\gamma$ is an integral curve of $\Gamma$, then so is the dilated curve $t \mapsto \lambda \tilde\gamma(\lambda t )$ for every $\lambda > 0$.
A curve of the form $\gamma = \pi \circ \tilde\gamma$, where $\tilde\gamma$ is an integral curve of $\Gamma$, is called a \emph{$\Gamma$-geodesic}.

\begin{example}\normalfont
	The flat spray on $\RR^n$ is the spray whose geodesics are straight lines, parametrized by constant speed. In a coordinate chart $(x^i,y^i)$ on $T\RR^n$, the flat spray is given by $\Gamma = y^i\partial_{x^i}$.
\end{example}
\begin{example}\normalfont
	More generally, the geodesic spray $\Gamma_g$ of a Riemannian manifold $(M,g)$ is the vector field generating the geodesic flow on $TM$. The $\Gamma_g$-geodesics are the constant-speed geodesics of the metric $g$. If $x^i$ are local coordinates on $M$, and $(x^i,y^i)$ are the induced local coordinates on $TM$, the geodesic spray is given in this local chart by $\Gamma_g = y^i\partial_{x^i} - \left(\Gamma^k_{ij}y^iy^j\right)\partial_{y_k}$, where $\Gamma^k_{ij}$ are the Christoffel symbols of $g$ in this chart. In case of a Riemannian surface, the geodesic spray is the vector field $E_1$.
\end{example}

We recall some definitions and facts about sprays. For more details see \cite{Sh}.
For each $v \in TM$ there exists a unique $\Gamma$-geodesic $\gamma_v$ satisfying $\dot\gamma_v(0) = v$, defined on a maximal open interval $I_v \subseteq \RR$. If $v = 0$ then $\gamma_v\equiv \pi(v)$ and $I_v = \RR$. Let 
	$$\cU : = \{v \in TM \mid 1 \in I_v\}.$$
	 Define the exponential map of $\Gamma$ by 
	$$\exp^\Gamma(v) : = \gamma_v(1), \qquad v \in \cU.$$

By the homogeneity of $\Gamma$, for every $v \in \cU$ and every $t \in I_v$,

\begin{equation}\label{exphom} \exp^\Gamma(tv) = \gamma_{v}(t).\end{equation}

For every $x \in M$ we set $\cU_x : = \cU\cap T_xM$ and define
$$\exp^\Gamma_x : = \exp^\Gamma\vert_{T_xM}.$$
\begin{theorem}[Whitehead \cite{Wh}, see also {\cite[Theorem 14.1.1]{Sh}}]\label{white}
	The exponential map $\exp^\Gamma$ is $C^1$ on $\cU$, and smooth away from the zero section.	For every $x \in M$, the differential of $\exp^\Gamma_x$ at $0$ is the identity (under the identification $T_0T_xM \cong T_xM$). 
\end{theorem}

As a consequence, $\exp^\Gamma_x$ is a $C^1$ diffeomorphism from a neighborhood of $0$ to a neighborhood of $x$. We also define the backwards exponential map by
$$\ola{\exp}^\Gamma(v) : = \gamma_v(-1)$$
on the open set $\ola{\cU} : = \{v \in TM \mid -1 \in I_v\}$, and define $\ola{\cU}_x$ and $\ola{\exp}_x^\Gamma$ similarly for all $x \in M$.

\begin{definition}[Simple spray]\normalfont
	A spray $\Gamma$ will be called \emph{simple} if for every $x \in M$ the maps $\exp^\Gamma_x$ and $\ola{\exp}^\Gamma_x$ are both $C^1$ diffeomorphisms from $\cU_x$ to $M$ and from $\ola{\cU}_x$ to $M$, respectively.
\end{definition}

\begin{example}[Simple sprays]\normalfont
	The flat spray in $\RR^n$ is simple. More generally, the geodesic spray of any Cartan-Hadamard manifold (i.e. complete, simply connected Riemannian manifold with nonpositive sectional curvature) is simple.
\end{example}

Let $(M,g)$ be a Riemannian surface and let $\Gamma$ be a spray on $M$. The semispray condition, together with \eqref{theta12def}, imply that $\theta_1(\Gamma) = 1$ and $\theta_2(\Gamma) = 0$. Thus every spray has the form
\begin{equation}\label{Gammaform}\Gamma = E_1 + {k}\,V + {h} \, R,\end{equation}
for some smooth functions ${k},{h}$ on $TM$. 

\medskip
In terms of covariant derivatives, if $\Gamma$ is given by \eqref{Gammaform}, the $\Gamma$-geodesics are exactly the solutions to the second-order ordinary differential equation
\begin{equation}\label{Gammaode} \nabla_{\dot\gamma}\dot\gamma = {h}(\dot\gamma)\,\dot\gamma + {k}(\dot\gamma)\,\dot\gamma^\perp.\end{equation}

Thus, if $\gamma$ is a $\Gamma$-geodesic, then the (signed) geodesic curvature of $\gamma$ with respect to $g$ is given by $\left<\nabla_{\dot\gamma}\dot\gamma,\dot\gamma^\perp\right>/|\dot\gamma|^3 = {k}(\dot\gamma)/|\dot\gamma|$.

\begin{definition}[Metric spray]\normalfont
	Let $(M,g)$ be a Riemannian surface and let $\Gamma$ be a spray on $M$. We say that $\Gamma$ is \emph{metric} if the geodesics of $\Gamma$ are constant-speed with respect to $g$, or equivalently if $\Gamma g = 0$. In terms of the representation \eqref{Gammaform}, the spray $\Gamma$ is metric if and only if ${h} \equiv 0$.
\end{definition}

\begin{definition}[Magnetic spray]\normalfont
	A spray $\Gamma$ on a Riemannian surface $(M,g)$ will be called \emph{magnetic} with respect $g$ if the function ${k}$ defined above is independent of the direction, i.e. $Vk = 0$; equivalently, there exists a function $\kappa:M \to \RR$ such that $k(v) = |v|\cdot {\kappa}(\pi(v))$ for all $v \in TM$. In this case we call $\kappa$ the \emph{geodesic curvature function} of $\Gamma$ (with respect to $g$). 
\end{definition}


If $\Gamma$ is metric and simple then for every $x \in M$ we may define a vector field $V_x$ on $M\setminus\{x\}$ as the pushforward of the vector field $R/|R|_g$ on $T_xM$ via the exponential map:
$$V_x : = \left(\exp_x^\Gamma\right)_*(R/|R|_g).$$
The integral curves of $R/|R|_g$ are lines through the origin in $T_xM$; by \eqref{exphom}, such lines are mapped by $\exp^\Gamma_x$ to unit-speed $\Gamma$-geodesics. It follows that the vector field $V_x$ is smooth on $M \setminus \{x\}$ and satisfies $|V_x|_g = 1$, and the integral curves of $V_x$ are unit-speed $\Gamma$-geodesics emanating from $x$. Thus by \eqref{Gammaode} we have

$$\nabla_{V_x}V_x = {k}(V_x)V_x^\perp$$
on $M \setminus \{x\}$. We also consider
\begin{equation}\label{etaxdef}\eta_x : = V_x^\flat.\end{equation}
	The one-form $\eta_x$ is defined and smooth on $M \setminus \{x\}$ and $|\eta_x|_g = |V_x|_g = 1$. Moreover
	\begin{align*}
		d\eta_x(V_x,V_x^\perp) & = V_x\eta_x(V_x^\perp) - V_x^\perp\eta_x(V_x) - \eta_x([V_x,V_x^\perp])\\
		& = - \eta_x([V_x,V_x^\perp])\\
		& = - \left<V_x,\nabla_{V_x}V_x^\perp - \nabla_{V_x^\perp}V_x\right>\\
		& = - \left<V_x,\nabla_{V_x}V_x^\perp\right>\\
		& = \left<\nabla_{V_x}V_x,V_x^\perp\right>\\
		& = {k}(V_x).
	\end{align*}
	In particular, if $\Gamma$ is magnetic with respect to $g$ and $\kappa$ is its geodesic curvature function, then
	\begin{equation}\label{detax} d\eta_x = \kappa\,\omega_g \qquad \text{ on } M\setminus\{x\}. \end{equation}

\begin{definition}[Convexity]\normalfont	
	Let $\Gamma$ be a spray on a manifold $M$ and let $A \subseteq M$. The set $A$ is said to be \emph{$\Gamma$-convex} if for every $\Gamma$-geodesic $\gamma$, if $\gamma(0) \in A$ and $\gamma(1) \in A$ then also $\gamma(t) \in A$ for every $t \in (0,1)$. The \emph{$\Gamma$-convex hull} of a subset $A \subseteq M$ is the smallest $\Gamma$-convex set containing $A$. 

	\medskip
	We shall say that a spray $\Gamma$ is \emph{proper} if the $\Gamma$-convex hull of every precompact set is precompact. This condition is not satisfied by all sprays, see Example \ref{noncchex} below.
\end{definition}

\begin{remark}\label{openchrmk}\normalfont
	If $\Gamma$ is simple, then the interior $A^\circ$ of a $\Gamma$-convex set $A$ is also $\Gamma$-convex. Indeed, the image of $A^\circ\times A^\circ\times(0,1)$ under the map $(x,y,t)\mapsto \exp^\Gamma_x(t(\exp^\Gamma_x)^{-1}(y))$ is open (since $\Gamma$ is simple), contains $A^\circ$ by Theorem \ref{white}, and is contained in $A$ since $A$ is $\Gamma$-convex, hence it equals $A^\circ$.
\end{remark}

\begin{example}[Improper spray]\normalfont\label{noncchex}
	Consider a spray on $\RR^2$ with the following property: for every $n \ge 0 $, the curve $\gamma_n(t) : =  (t,-\cos(3^n t) + 2n)$, $t \in [-3^{-n}\cdot2\pi,3^{-n}\cdot2\pi]$, is a $\Gamma$-geodesic, as well as the line $\ell_n(t) : = (-t , 2n + 1)$, $t \in \RR$, see Figure \ref{noncchfig}. Note that since $\gamma_n$ and $\ell_n$ move in opposite directions, there is no contradiction to the uniqueness of $\Gamma$-geodesics with given initial conditions. One can even take $\Gamma$ to be a simple spray. The $\Gamma$-convex hull of the curve $\gamma_0$ contains all of the curves $\gamma_n$, and is therefore not compact.
\end{example}

\begin{figure}
	\centering

	\includegraphics[width = .8\textwidth]{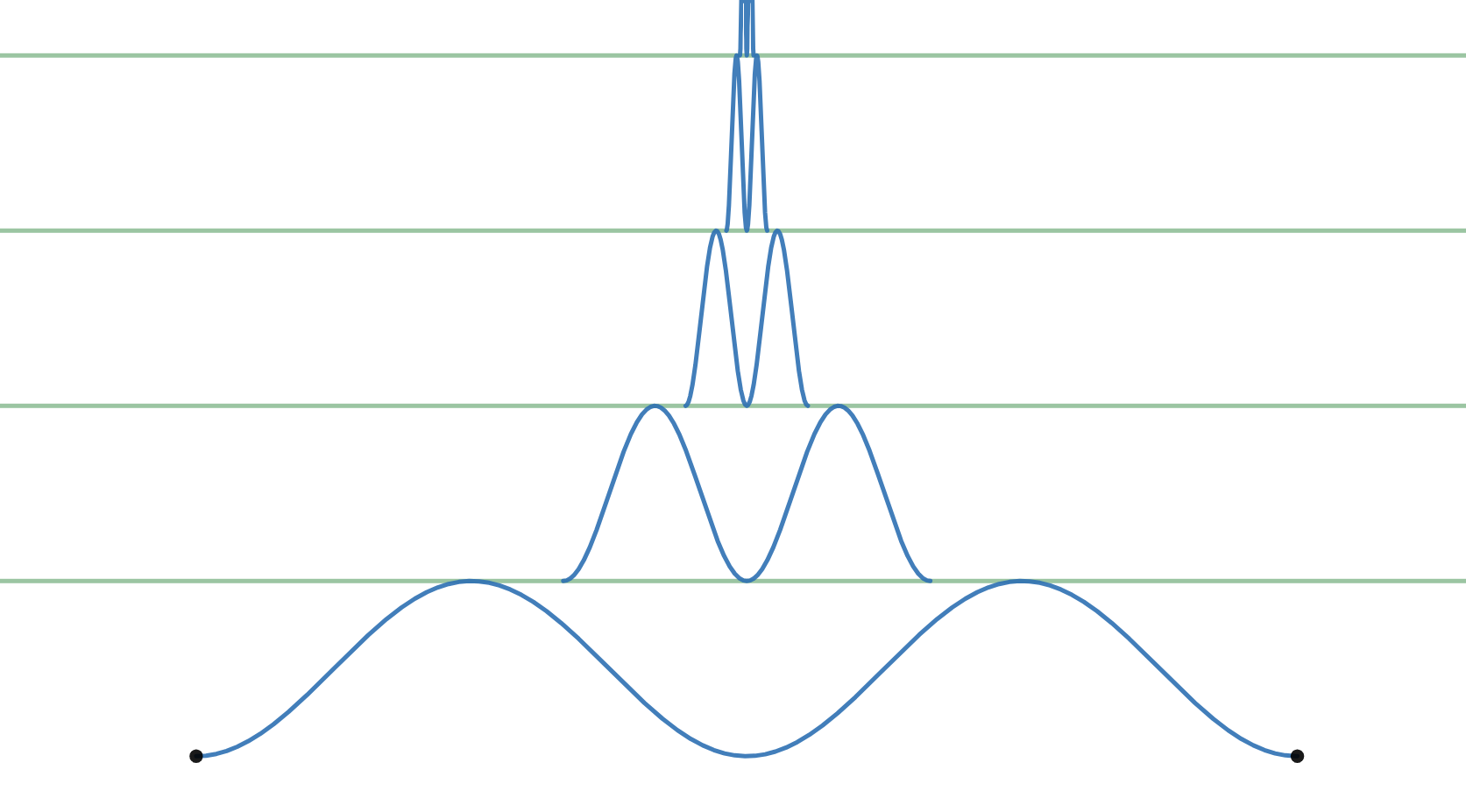}

	\caption{An improper spray}
	\label{noncchfig}
\end{figure}

\subsection{Jacobi Fields}

Let $\Gamma$ be a spray on a manifold $M$ and denote its flow on $TM$ by $\Phi_t$. Let $\gamma$ be a $\Gamma$-geodesic. A vector field $S$ along $\gamma$ is called a \emph{$\Gamma$-Jacobi field} if it satisfies one of the following equivalent conditions:
\begin{enumerate}
	\item $S(t) = d\pi(d\Phi_t(\xi))$ where $\xi \in TTM$ satisfies $\pi(\xi) = \dot\gamma(0)$ and $d\pi(\xi) = S(0)$.
	\item There exists a variation $F(s,t)$ through $\Gamma$-geodesics such that $S = dF\left(\partial_s\vert_{s=0}\right)$.
\end{enumerate}
Here, by a variation through $\Gamma$-geodesics we mean that $F$ is smooth and $F(s,\cdot)$ is a $\Gamma$-geodesic for every $s$. Let us prove the equivalence of 1 and 2. Suppose first that such $F$ exists. Recall that $\Gamma$-geodesics are curves of the form $\pi\circ\tilde\gamma$ where $\tilde\gamma$ is an integral curve of $\Gamma$. Thus by \eqref{exphom} we can write
\begin{equation}\label{Vardef}F(s,t) = \exp^\Gamma(tc(s)) = \pi(\Phi_t(c(s))),\end{equation}
where $c$ is a curve on $TM$ with $c(0) = \dot\gamma(0)$. Write $\xi = \dot c(0)$. Then $\pi(\xi) = \pi(\dot c(0)) = c(0) = \dot\gamma(0)$. By differentiating with respect to $s$ we get $S(t) = dF\left(\partial_s(t,0)\right) = d(\pi\circ\Phi_t)(\dot c(0)) = d\pi(d\Phi_t(\xi))$, and by setting $t=0$ we see that $S(0) = d\pi(\xi)$. In the other direction, if $S = d\pi(d\Phi_t(\xi))$, then we can take any curve $c$ on $TM$ with $c(0) = \dot\gamma(0)$ and $d\pi(\dot c(0)) = S(0)$ and define $F$ by \eqref{Vardef}.

\begin{definition}[Transversal Jacobi field]\normalfont
	Let $\Gamma$ be a spray on a two-dimensional oriented manifold $M$. We shall say that a $\Gamma$-Jacobi field $S$ along a $\Gamma$-geodesic $\gamma$ is \emph{transversal} if $\omega(\dot\gamma,S) > 0$, where $\omega$ is any volume form on $M$.
\end{definition}

\subsection{Projective Finsler metrizability}

A \emph{Finsler metric} on a manifold $M$ is a function $\cF : TM \to [0,\infty)$, positive and smooth away from the zero section, which satisfies the following
 requirements:

\begin{itemize}
	\item \emph{positive homogeneity}: $\cF(\lambda v ) = \lambda \cF(v)$ for all $v \in TM$ and $\lambda>0$.
	\item \emph{strong convexity}: Fix $x \in M$. Then the function $\cF^2$ is convex in the linear space $T_x M$, and moreover  its Hessian at any point $0 \neq v \in T_x M$ is positive definite.
\end{itemize}
A manifold endowed with a Finsler metric is called a \emph{Finsler manifold}. A Finsler metric induces a metric on $M$ by setting
$$d(x,y) = \inf_\gamma \mathrm{Length}(\gamma),$$
where
$$\mathrm{Length}(\gamma) : = \int_0^1\cF(\dot\gamma(t)) dt,$$
and where the infimum is taken over all $C^1$ curves joining $x$ to $y$.

\medskip
A \emph{minimizing geodesic} of a Finsler metric is a constant-speed curve $\gamma:[a,b] \to M$ satisfying $\mathrm{Length}(\gamma) = d(\gamma(a),\gamma(b))$.  A {\it geodesic} is a curve
$\gamma:[a,b] \to M$ that is locally a minimizing geodesic, i.e., for any $t_0 \in [a,b]$ there exists $\delta > 0$ such that the restriction of $\gamma$
to the interval $[a,b] \cap (t_0 - \delta, t_0 + \delta)$
is a minimizing geodesic. Equivalently, a geodesic is a solution to the Euler-Lagrange equation associated with the Lagrangian $\cF^2/2$. The \emph{geodesic spray} of $\cF$ is the spray on $TM$ whose geodesics are the constant-speed geodesics of the metric $\cF$. We say that $(M,\cF)$ is \emph{geodesically convex} if any two points  $x,y \in M$ are joined by a minimizing geodesic. We refer the reader to \cite{BCS} for more background on Finsler metrics. 

\begin{definition}[Projectively equivalent sprays]\normalfont
	Two sprays $\Gamma, \tilde \Gamma$ are said to be \emph{projectively equivalent} if there exists a scalar function $f:TM\to\RR$ such that $\Gamma - \tilde\Gamma = fR$. Equivalently, the geodesics of $\Gamma$ and $\tilde\Gamma$ coincide up to orientation-preserving reparametrization.
\end{definition}

\begin{definition}[Projectively Finsler-metrizable spray]\normalfont
	A spray $\Gamma$ on a manifold $M$ is said to be \emph{projectively Finsler-metrizable (pFm)} if there exists a geodesically convex Finsler metric $\cF$ on $M$ such that $\Gamma$ is projectively equivalent to the geodesic spray of $\cF$, that is, if the geodesics of the Finsler metric $\cF$ coincide with the geodesics of $\Gamma$ up to orientation-preserving reparametrization. We remark that the usual definition of projective Finsler metrizability does not include geodesic convexity of $\cF$. See \cite{BM} and references therein for more information on projective Finsler metrizability. See Darboux \cite{Dar} and Matsumoto \cite{Mat} for a local resolution of the projective metrizabilty problem in two dimensions (here `local' refers to locality also in the tangent space).
\end{definition}

\begin{proposition}\label{metrizabilityppn}
	Let $M$ be a two-dimensional manifold and let $\Gamma$ be a simple proper spray on $M$. Suppose that there exists a Riemannian metric $g$ on $M$ such that $\Gamma$ is magnetic with respect to $g$. Then for every compact set $A \subseteq M$ there exists an open $\Gamma$-convex set $U \supseteq A$ such that the restriction of $\Gamma$ to $U$ is projectively Finsler-metrizable.
\end{proposition}
\begin{proof}
	Let $g$ be a Riemannian metric such that $\Gamma$ is magnetic with respect to $g$, and let $\kappa : M \to \RR$ be the corresponding geodesic curvature function. By replacing $\Gamma$ with a projectively equivalent spray we may assume that $\Gamma$ is metric with repsect to $g$. Since $\Gamma$ is proper, there exists a precompact open set $U$ containing $A$ which is $\Gamma$-convex (see Remark \ref{openchrmk}).

	\begin{lemma}\label{etalemma}
		There exists a 1-form $\eta$ on $U$ such that $|\eta|_g < 1$ and $d\eta = \kappa\omega_g$ on $U$.
	\end{lemma}
	Let us first finish the proof assuming Lemma \ref{etalemma}. Let $\eta$ be the 1-form from Lemma \ref{etalemma} and define a Finsler metric $\cF$ on $U$ by
	$$\cF : = \sqrt{g} - \eta.$$
	Since $|\eta|_g < 1$, this is indeed a Finsler metric, of \emph{Randers} type (see e.g. \cite{BCS}). Since $\Gamma$ is simple and $U$ is $\Gamma$ convex, every pair of points in $U$ is joined by a $\Gamma$-geodesic. We shall prove that this $\Gamma$-geodesic is uniquely 	length-minimizing with respect to $\cF$. It will then follow that $\Gamma$-geodesics coincide with the geodesics of $\cF$ up to orientation-preserving reparametrization, and that $\cF$ is geodesically convex. Let $p,q \in U$, let $\gamma$ be a $C^1$ curve in $U$ joining $p$ to $q$ and let $\gamma_0 : [0,1] \to U$ be the $\Gamma$-geodesic joining $p$ to $q$ (which lies inside $U$ since $U$ is $\Gamma$-convex). 

	\medskip
	Since $\Gamma$ is simple, the $\Gamma$-geodesic $\gamma_0$ can be extended to a geodesic $\bar\gamma:I \to M$, for some interval $I$ containing $[0,1]$, so that $\bar\gamma\vert_{[0,1]} = \gamma_0$, and there exists some $t<0$ such that  $\gamma$ and $\gamma_0$ are homotopic in $M \setminus\{\bar\gamma(-t)\}$. Set $x = \bar\gamma(-t)$, and let $\eta_x$ be defined by \eqref{etaxdef}. See Figure \ref{pqx}. By \eqref{detax} we have $d\eta_x = \kappa\omega_g = d\eta$ on $M\setminus\{x\}$, and therefore $\int_{\gamma_0}(\eta - \eta_x) = \int_{\gamma}(\eta - \eta_x)$. Hence, in order to prove that $\gamma_0$ is shorter than $\gamma$ with respect to $\cF$, it  suffices to prove that 
	\begin{equation}\label{gammagamma0lengths}\int_{\gamma_0}(\sqrt{g} - \eta_x) \le \int_{\gamma}(\sqrt{g} - \eta_x).\end{equation}
	Since $\gamma_0$ is part of the $\Gamma$-geodesic joining $x$ to $q$, its tangent is proportional to $V_x$ and therefore $\eta_x(\dot\gamma_0) = |\dot\gamma_0|$. Thus the left hand side vanishes, while the right hand side is nonnegative since $|\eta_x|_g = 1$. Moreover, equality implies that $\dot\gamma$ is proportional to $V_x$, whence $\gamma$ coincides with $\gamma_0$ up to orientation-preserving reparametrization.
\end{proof} 	

\begin{figure}
	\centering

	\includegraphics[width = .8\textwidth]{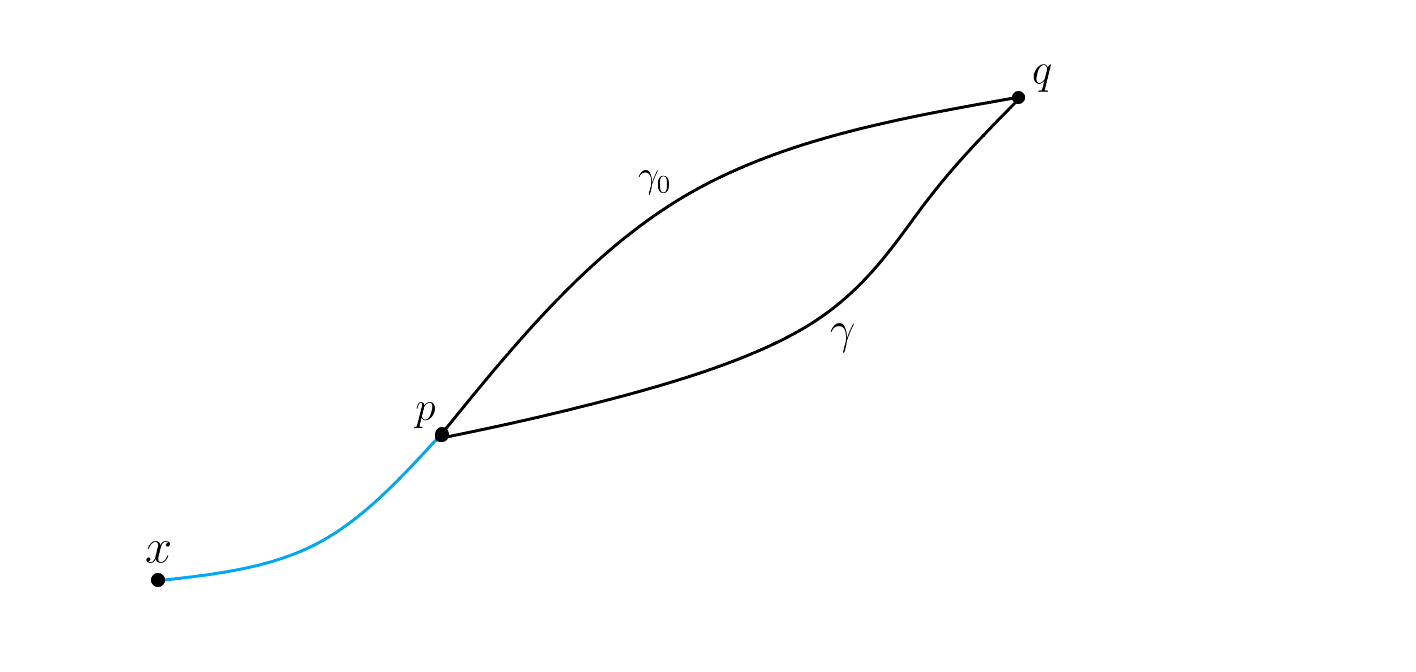}

	\caption{Proof of Proposition \ref{metrizabilityppn}}
	\label{pqx}
\end{figure}

\begin{proof}[Proof of Lemma \ref{etalemma}]
	Choose $x,y,z \in M\setminus U$ which do not lie on a common $\Gamma$-geodesic. Define $\eta = \frac13(\eta_x + \eta_y + \eta_z)$, with $\eta_x,\eta_y,\eta_z$ defined as in \eqref{etaxdef}. Since $x,y,z \notin U$, the 1-form $\eta$ is smooth on $U$, and by \eqref{detax} we have $d\eta = \kappa\omega_g$. Finally, since $x,y,z$ do not lie on the same $\Gamma$-geodesic, the one-forms $\eta_x,\eta_y,\eta_z$ do not all coincide at any point, and therefore $|\eta| < 1$.
\end{proof}

\begin{remark}
	In the last proof we constructed a solution to the linear equation $d\eta = \kappa\omega_g$, under the constraint $|\eta| \le 1$. If we set $X := (\eta^\sharp)^\perp$, then this is equivalent to the equation $\mathrm{div}X = \kappa$ subject to the constraint $|X| \le 1 $. There are several ways to solve this equation, such as stipulating $X = \nabla f$ and solving the Poisson equation $\Delta f = \kappa$. However, in order to satisfy the requirement $|X| \le 1$, rather than using linear methods, we instead solved the nonlinear equation $\nabla_VV = \kappa V^\perp$ under the constraint $|V| = 1$, where we set $V := - X^\perp = \eta^\sharp$. Note that indeed $\mathrm{div} X = \mathrm{div} V^\perp = \kappa$.
\end{remark}

\section{Weighted spray spaces}\label{weightedsec}

A triple $(M,\Gamma,\omega)$ where $M$ is a manifold, $\Gamma$ is a spray and $\omega$ is a volume form will be called a \emph{weighted spray space}. We shall only deal with the two-dimensional case. We denote the Lie derivative with respect to a vector field $X$ by $\cL_X$, and write $\cL_X^2 := \cL_X\cL_X$. The interior product is denoted by $\iota$.

\begin{proposition}\label{nncprop}
	 Let $(M,\Gamma,\omega)$ be a weighted spray space, $\dim M = 2$. Then the following are equivalent:
	 \begin{enumerate}
	 		\item For every $\Gamma$-geodesic $\gamma$ and every transversal $\Gamma$-Jacobi field $S$ along $\gamma$, the function $J(t) : = \omega(\dot\gamma(t),S(t))$ is concave.
	 		\item There exists a nonnegative smooth function $Q$ on $TM$ such that
	 		\begin{equation}
	 			\cL^2_\Gamma(\iota_\Gamma\pi^*\omega) = - Q \, \cdot \, \iota_\Gamma\pi^*\omega.
	 		\end{equation}
	 \end{enumerate}
\end{proposition}
\begin{definition}[Nonnegatively curved weighted spray space]\normalfont
	If any of the equivalent conditions in Proposition \ref{nncprop} holds we shall say that the weighted spray space $(M,\Gamma,\omega)$ is \emph{nonnegatively curved}.
\end{definition}
\begin{proof}
	It suffices to prove that condition 2 is equivalent to  
	\begin{itemize}\it{
		\item[1'.] for every $\Gamma$-geodesic $\gamma$ and every transversal $\Gamma$-Jacobi field $S$ along $\gamma$ we have $J''(0) \le 0$.}	
	\end{itemize} 
	Indeed, the function $J$ is always smooth for a transversal Jacobi field, so this condition is clearly weaker than condition 1, and on the other hand, if condition 1 does not hold for some $\Gamma$-Jacobi field, then by time translation we can find a $\Gamma$-Jacobi field for which $J''(0) > 0 $.

	\medskip
	Let $\gamma$ be a $\Gamma$-geodesic and let $S$ be a Jacobi field along $\gamma$. Let $\xi \in TTM$ satisfy $\pi(\xi) = \dot\gamma(0)$ and $d\pi(\xi) = S(0)$, so that $S(t) = d\pi(d\Phi_t(\xi)).$ Note also that by the semispray condition $\dot\gamma(t) = \pi\left(\Gamma\vert_{\dot\gamma(t)}\right) = d\pi\left(\Gamma\vert_{\dot\gamma(t)}\right)$.
	Thus
	\begin{align*}
		J(t) & = \omega(\dot\gamma(t),S(t)) \\
		& = \omega(d\pi(\Gamma\vert_{\dot\gamma(t)}), d\pi(d\Phi_t(\xi))) \\
		& = \pi^*\omega(\Gamma\vert_{\dot\gamma(t)},d\Phi_t(\xi))\\
		& = \iota_\Gamma\pi^*\omega(d\Phi_t(\xi))\\
		& = \Phi_t^*\left(\iota_\Gamma\pi^*\omega\right)(\xi).
	\end{align*}
	Since $\Phi_t$ is the flow of $\Gamma$, it follows that
	$$J(0) = \iota_\Gamma\pi^*\omega(\xi) \qquad \text{ and } \qquad J''(0) = \cL_\Gamma^2(\iota_\Gamma\pi^*\omega)(\xi).$$
	Note that if $S$ is transversal then $\iota_\Gamma\pi^*\omega(\xi) > 0$. Indeed, by setting $t=0$ in the above calculation we see that $\omega(\dot\gamma(0),S(0)) = \iota_\Gamma\pi^*\omega(\xi)$. Thus condition 1' is equivalent to the assertion that for every $\xi \in TTM$ such that $\iota_\Gamma\pi^*\omega(\xi) > 0 $, it also holds that $\cL_\Gamma^2(\iota_\Gamma\pi^*\omega)(\xi) \le 0$. Since both expressions are linear in $\xi$, this is equivalent to condition 2.
\end{proof}
\begin{lemma}\label{Kgammalemma}
	Let $(M,g)$ be a Riemannian surface, let $\Gamma = E_1 + {k} V + {h} R$ be a spray on $TM$, and let $\omega = e^{-\vphi}\omega_g$ be a volume form on $M$. Then 
	\begin{equation}\label{lieeq}\cL^2_\Gamma\left(\iota_\Gamma\pi^*\omega\right) = - Q \cdot \iota_\Gamma\pi^*\omega + \alpha\end{equation}
	where
	$$\alpha := (V{k} - 2d \vphi - {h})e^{-\vphi}(\psi - {k}\theta_1)$$
	and
	\begin{equation}\label{Qdef}Q := K + {k}^2 - {h}^2 +\Gamma(d\vphi) - (d\vphi)^2 + \Gamma{h} - E_2{k} - 2{h} d\vphi .\end{equation}
\end{lemma}
\begin{remark}\normalfont
	In \eqref{lieeq}, and in similar formulae below, the one-form $d\vphi$ is regarded as a function on $TM$, and the term $\Gamma(d\vphi)$ denotes its derivative with respect to the vector field $\Gamma$.
\end{remark}
\begin{proof}
	By \eqref{omegagpullback} we have 
	$$\pi^*\omega = e^{-\vphi}\pi^*\omega_g = e^{-\vphi}\theta_1\wedge\theta_2.$$
	 Since $\theta_1(\Gamma) = 1$ and $\theta_2(\Gamma) = 0$, it follows that
	$$\iota_\Gamma\pi^*\omega = e^{-\vphi}\theta_2.$$
	\medskip
	From \eqref{structureTM} and Cartan's formula for a Lie derivative $\cL_X = \iota_Xd + d\iota_X$ we get
	\begin{align}\begin{split}\label{lieder1}
		\cL_\Gamma\theta_1 & = \rho - {h}\theta_1 + {k}\theta_2, \qquad \cL_\Gamma\theta_2 = -{h}\theta_2 + \psi  -{k}\theta_1, \qquad \cL_\Gamma\rho = d{h}, 	\end{split}\\&\text{and}\nonumber\\
	 	\cL_\Gamma\psi & = d{k} - K \theta_2 \nonumber\\
	 	 & = (E_1{k})\theta_1 + (E_2{k} - K)\theta_2 + (V{k})\psi + (R{k})\rho.\nonumber
	\end{align}
	Observe that 
	\begin{equation}\label{Rkappa} R{k} = R\psi(\Gamma) = d\psi(R,\Gamma) + \Gamma\psi(R)+ \psi([R,\Gamma]) =  \psi(\Gamma) =  {k},\end{equation}
	where in the third passage we have used \eqref{structureTM} and the fact that $(\theta_1,\theta_2,\psi,\rho)$ is dual to $(E_1,E_2,V,R)$ to show that the first two terms vanish, and then used the homogeneity $[R,\Gamma] = \Gamma$. Thus
	\begin{equation}\label{lieder2}
		\cL_\Gamma\psi = (E_1{k})\theta_1 + (E_2{k} - K)\theta_2 + (V{k})\psi + {k}\rho.
	\end{equation}
	By \eqref{lieder1}, \eqref{lieder2}, \eqref{Gammaform} and \eqref{Rkappa} we have
	\begin{align*}
		\cL_\Gamma^2\theta_2 = &  - (\Gamma{h})\theta_2 - {h}\cL_\Gamma\theta_2 + \cL_\Gamma\psi - (\Gamma{k})\theta_1 - {k}\cL_\Gamma\theta_1\\
		 = & - (\Gamma{h})\theta_2 + {h}^2\theta_2 - {h}\psi + {h}{k}\theta_1 + (E_1{k})\theta_1 + (E_2{k} - K)\theta_2 + (V{k})\psi +{k}\rho \\ & - (E_1{k} + {k} (V {k}) + {h}(R{k}))\theta_1 -{k}\rho + {k}{h}\theta_1- {k}^2\theta_2\\
		= & (-\Gamma{h} + {h}^2 +  E_2{k} - K - {k}^2)\theta_2 + (V{k} - {h})(\psi - {k} \theta_1).
	\end{align*}
	If we view $\vphi$ as a function on $TM$ which is constant on each fiber, then we have $\Gamma \vphi = d\vphi$ on $TM$. Thus
	\begin{align*}
		\cL^2_\Gamma(\iota_\Gamma\pi^*\omega) & = \cL^2_\Gamma(e^{-\vphi}\theta_2)\\
		= &  (\Gamma^2 e^{-\vphi})\theta_2 + 2(\Gamma e^{-\vphi})\cL_\Gamma\theta_2 + e^{-\vphi}\cL^2_\Gamma\theta_2\\
		= & \left(\Gamma^2 e^{-\vphi} + e^{-\vphi}(-\Gamma{h} + {h}^2 + E_2{k} - K - {k}^2) - 2{h}\Gamma e^{-\vphi}\right)\theta_2 \\
		& + (2\Gamma e^{-\vphi} + e^{-\vphi}(V{k} - {h}))(\psi - {k}\theta_1)\\
		= & \left(-\Gamma(d\vphi) + (d\vphi)^2 - \Gamma{h} + {h}^2 + E_2{k} - K - {k}^2 + 2{h} d\vphi\right) \cdot e^{-\vphi}\theta_2 \\
		& + \left(-2d\vphi + V{k} - {h}\right)e^{-\vphi}(\psi - {k}\theta_1)\\
		& = - Q \cdot \iota_\Gamma\pi^*\omega + \alpha 
	\end{align*}
	as claimed.
\end{proof}
\begin{corollary}\label{nncriemcor}
	Let $(M,g)$ be a Riemannian surface, let $\Gamma = E_1 + {k} V$ be a metric spray on $(M,g)$ and let $\omega = e^{-\vphi}\omega_g$ be a volume form on $M$. Then the following are equivalent:
	\begin{enumerate}
		\item The weighted spray space $(M,\Gamma,\omega)$ is nonnegatively curved.
		\item The spray $\Gamma$ is magnetic with respect to the metric $e^{-4\vphi}g$, and 
		\begin{equation}\label{nncRiemeq}K + {k}^2  +\Gamma(d\vphi) - (d\vphi)^2 - E_2{k} \ge 0.\end{equation}
	\end{enumerate}
	In particular, the weighted spray space $(M,\Gamma,\omega_g)$ is nonnegatively curved if and only if $\Gamma$ is magnetic with respect to $g$ and 
	$$K + \kappa ^ 2 - |\nabla\kappa| \ge 0,$$
	where $\kappa$ is the geodesic curvature function of $\Gamma$.
\end{corollary}

\begin{remark}
	In the case $\vphi \equiv 0$, the quantity $K + {k}^2 - E_2{k}$ on the left hand side of \eqref{nncRiemeq} coincides with the Ricci scalar of the spray $\Gamma$ as defined in e.g. \cite{Sh}.
\end{remark}

\begin{proof}
	Since $\Gamma$ is metric, ${h} \equiv 0$. By Proposition \ref{nncprop} and Lemma \ref{Kgammalemma}, the weighted spray space $(M,\Gamma,\omega)$ is nonnegatively curved if and only if $V{k} - 2d \vphi = 0$  and inequality \eqref{nncRiemeq} holds.  By Lemma \ref{conformalspraylemma},  
	$$\Gamma = E_1 + {k} V =  \tilde E_1 + ({k} - 2\star d\vphi) V - 2 d\vphi \, R,$$ 
	where $\tilde E_1$ is the geodesic spray of $\tilde g = e^{-4\vphi}g$. Thus $\Gamma$ is magnetic with respect to $\tilde g$ if and only if 
	$$0 = V({k} - 2 \star d\vphi) = V{k} - 2d\vphi,$$
	as desired. Here the second equality holds true because, if $(x,y)$ are normal coordinates at $p \in M$ and $(x,y,u,v)$ are the corresponding canonical local coordinates on $TM$, then at $p$ we have
	$$V(\star d\vphi) = (-v\partial_u + u \partial_v)(-\vphi_yu + \vphi_xv) = v \vphi_y + u\vphi_x = d\vphi.$$

	\medskip
	Suppose that $\vphi \equiv 0$. Then $(M,\Gamma,\omega = \omega_g)$ is nonnegatively curved if and only if $\Gamma$ is magnetic with respect to $g$, and $K + {k}^2 - E_2{k} \ge 0$. If $\Gamma$ is magnetic then $V{k} = 0$, and therefore, writing ${k} = \sqrt{g}\cdot \kappa\circ\pi$ for a function $\kappa : M \to \RR$, we have $E_2{k} = [V,E_1]{k} = VE_1{k} - E_1V{k} = V(d\kappa) = - \star d\kappa$. It follows that $(M,\Gamma,\omega_g)$ is nonnegatively curved if and only if $M$ is magnetic and the function $\kappa$ satisfies $0 \le K + \kappa^2 - |\star \, d\kappa| = K + \kappa^2 - |\nabla\kappa|$.
\end{proof}

\begin{remark}\label{CDrmk}
	In the spirit of curvature-dimension theory of Bakry-Emery \cite{BGL} and Lott-Sturm-Villani \cite{LV,sturm}, one can extend the notion of a nonnegatively curved weighted spray space on a surface to that of a $CD(r,N)$ weighted spray space for any $r \in \RR$ and $N \ne 1$, in which the requirement $J'' \le 0 $ (with $J$ defined as in Proposition \ref{nncprop}) is replaced by the condition
	\begin{equation}\label{CDKNeq}\frac{J''}{J} - \left(\frac{N-2}{N-1}\right)\left(\frac{J'}{J}\right)^2 + r \le 0\end{equation}
	(see Theorem \ref{sprayBMCDthm} below for the motivation for this definition in the case $r=0$). Let $\Gamma$ be a spray on a Riemannian surface $(M,g)$ and suppose for simplicity that $\Gamma$ is metric. Let $\omega = e^{-\vphi}\omega_g$ be a volume form on $M$. Let $TTM^+:=\{\xi \in TTM \mid \theta_2(\xi) > 0\}$ and define a function $z : TTM^+ \to \RR$ by $$z(\xi) : = \frac{\psi(\xi) - {k}\theta_1(\xi)}{\theta_2(\xi)}.$$ 
	A straightforward modification of Proposition \ref{nncprop} and Lemma \ref{Kgammalemma} and their proofs shows that \eqref{CDKNeq} holds for every transversal $\Gamma$-Jacobi field if and only if
	\begin{equation}\label{sprayCDKNpoly}-Q + Bz -\frac{N-2}{N-1}(z - d\vphi)^2 + r \le 0 \qquad \text{ on } TTM^+,\end{equation}
	where 
	$$Q : = K + {k}^2  +\Gamma(d\vphi) - (d\vphi)^2  - E_2{k} \qquad \text{ and } \qquad B : = V{k} - 2d\vphi.$$
	Note that $Q$ and $B$ are functions on $TM$. Since, by linear independence of the one forms $\psi-{k}\theta_1$ and $\theta_2$, the function $z$ can attain any value on each fiber of $TTM^+$, the expression on the left hand side of \eqref{sprayCDKNpoly}, viewed as a polynomial in $z$ whose coefficients are functions on $TM$, must be nonnegative for all values of $z$. For $N\notin [1,2]$ this is equivalent to
	$$Q - d\vphi\cdot B - \frac{N-1}{N-2}\cdot\frac{B^2}{4} \ge r \qquad \text{ on } TM,$$
	or
	$$K + {k}^2 + \Gamma(d\vphi) - E_2{k} - \frac{\left((N-1)V{k} - 2\,d\vphi\right)^2}{4(N-1)(N-2)} - \frac{(d\vphi)^2}{N-1} \ge r \qquad \text{ on } TM.$$
	We remark that when $\Gamma$ is the geodesic spray, i.e. ${k} = 0$, the expression on the left hand side equals the generalized Ricci curvature of the weighted Riemannian manifold $(M,g,\omega)$.
\end{remark}

\section{Brunn-Minkowski for sprays}\label{BMsec}

Fix a two-dimensional weighted spray space $(M,\Gamma,\omega)$. Given two subsets $A,B \subseteq M$ and $0 < \lambda < 1$, we define 
$$\cM_\Gamma(A,B;\lambda) := \left\{\gamma(\lambda) \mid \gamma \text{ is a $\Gamma$-geodesic }, \, \gamma(0) \in A, \, \gamma(1) \in B\right\}.$$

For example, if $M = \RR^2$ and $\Gamma$ is the flat spray, then $\cM_\Gamma(A,B;\lambda) = (1-\lambda)A + \lambda B$, where $+$ denotes Minkowski summation. If $M$ is a Riemannian surface and $\Gamma$ is its geodesic spray, then $\cM_\Gamma(A,B;\lambda)$ is similar to the operation defined in \cite{CMS,sturm}, except that in our definition we do not require the geodesics in the definition of $\cM_\Gamma(A,B;\lambda)$ to be minimizing.

\medskip
Denote by $\mu$ the unique Borel measure on $M$ which satisfies 
$$\mu(A) = \int_A\omega$$ for every open set $A \subseteq M$. In this section we characterize, under some assumptions on the spray, two-dimensional weighted spray spaces $(M,\Gamma,\omega)$ satisfying the Brunn-Minkowski inequality
	\begin{equation}\label{sprayBM}\tag{BM}
			\mu(\cM_\Gamma(A,B;\lambda))^{1/2} \ge (1-\lambda)\cdot\mu(A)^{1/2} + \lambda\cdot\mu(B)^{1/2}
	\end{equation}
for every nonempty Borel subsets $A,B \subseteq M$.
\begin{theorem}\label{sprayBMthm}
	Let $(M,\Gamma,\omega)$ be a simple, nonnegatively curved, two-dimensional weighted spray space. Suppose that $\Gamma$ is projectively Finsler-metrizable. Then \eqref{sprayBM} holds for every nonempty Borel subsets $A,B \subseteq M$ of positive measure and every $0<\lambda<1$.
\end{theorem}

The central tool in the proof of Theorem \ref{sprayBMthm} is Theorem \ref{needledecomp} below, which is a needle decomposition theorem for sprays, similar to (and generalizing) the horocyclic needle decomposition theorem \cite[Theorem 3.1]{AK}. 

\begin{definition}[Jacobi needle]\normalfont\label{jacneedledef}
	Let $M$ be a two-dimensional oriented manifold, let $\Gamma$ be a spray on $M$ and let $\omega$ be a volume form on $M$. Let $\gamma = \gamma(t)$ be a $\Gamma$-geodesic. A measure $\mu$ on $M$ will be called a \emph{$(\Gamma,\omega)$-Jacobi needle along $\gamma$} if there exists a transversal $\Gamma$-Jacobi field $S$ along $\gamma$ such that 
	\begin{equation}\label{jacneedleform}\mu = \gamma_\#\left(J(t)dt\right),\end{equation}
	where $$J(t) : = \omega(\dot\gamma(t),S(t))$$ and $\#$ denotes pushforward. A Dirac mass (i.e. a measure supported on a single point) is also considered a $(\Gamma,\omega)$-Jacobi needle.
\end{definition}
Intuitively, a $(\Gamma,\omega)$-Jacobi needle should be thought of as the restriction of $\omega$ to an infinitesimally thin strip made out of $\Gamma$-geodesics. It is intuitively clear that the notion of a $(\Gamma,\omega)$-Jacobi needle depends only on the projective class of $\Gamma$. Let us prove this fact, which will be useful for us.
\begin{lemma}\label{jacprojlemma}
	Let $\omega$ be a volume form on a two-dimensional manifold $M$ and let $\Gamma$ and $\tilde \Gamma$ be projectively-equivalent sprays on $TM$. Then every $(\Gamma,\omega)$-Jacobi needle is also a $(\tilde\Gamma,\omega)$-Jacobi needle.
\end{lemma}
\begin{proof}
	Suppose that $\mu$ is a $(\Gamma,\omega)$-Jacobi needle. We may assume that $\mu$ is not a Dirac mass (because then the statement is trivial), whence it takes the form \eqref{jacneedleform}, with $\gamma,S,J$ as in Definition \ref{jacneedledef}. Let $F(s,t)$ be a variation through $\Gamma$-geodesics such that $S = F_*(\partial_s\vert_{s=0})$. Since $\Gamma$ and $\tilde \Gamma$ are projectively-equivalent, their geodesics differ by an orientation-preserving reparametrization. Thus there exists a smooth function $t = t(s,\tau)$, strictly increasing in $\tau $, such that the following holds: if we set
	\begin{equation}
		\tilde F(s,\tau) = F(s,t(s,\tau))
	\end{equation}
	then the curve
		$$\tau \mapsto \tilde F(s,\tau)$$
	is a $\tilde\Gamma$-geodesic for every $s$. In particular, 
	$$\tilde\gamma(\tau) = \tilde F(0,\tau) = F(0,t(0,\tau)) = \gamma(t(0,\tau))$$
	is a $\tilde\Gamma$-geodesic, and the vector field 
	$$\tilde S = \tilde F_*(\partial_s\vert_{s=0})$$
	 along $\tilde\gamma$ is a $\tilde\Gamma$-Jacobi field induced by the variation $\tilde F$. By the chain rule,
		$$\tilde S = F_*(\partial _s\vert_{s=0}) + \frac{\partial t}{\partial s}\cdot F_*(\partial_t\vert_{s=0}) = S + \frac{\partial t}{\partial s} \cdot \dot\gamma.$$
		Thus
		$$\omega(\dot{\tilde\gamma},\tilde S) = \omega((\partial t/\partial \tau)\dot\gamma,S + (\partial t/\partial s)\dot\gamma) = (\partial t/\partial\tau)\omega(\dot\gamma,S)$$
		and therefore
		$$\mu = \gamma_\#(\omega(\dot\gamma,S)dt) = \tilde\gamma_\#((\partial t/\partial \tau)\omega(\dot\gamma,S)d\tau) = \tilde\gamma_\#(\omega(\dot{\tilde\gamma},\tilde S)d\tau).$$
	It follows that $\mu$ is a $(\tilde\Gamma,\omega)$-Jacobi needle.
\end{proof}
\begin{theorem}[Needle decomposition for pFm sprays]\label{needledecomp}
	Let $M$ be a two-dimensional manifold, let $\Gamma$ be a projectively Finsler-metrizable spray on $M$ and let $\omega$ be a volume form on $M$. Let $\rho_1, \rho_2: M \rightarrow [0, \infty)$ be compactly-supported measurable functions with
	$$ \int_M \rho_1 \omega = \int_M \rho_2 \omega < \infty. $$
	Then there is a collection $\Lambda$ of disjoint $\Gamma$-geodesics, a measure $\nu$ on $\Lambda$ and a family $\{\mu_\gamma\}_{\gamma \in \Lambda}$ of Borel measures on $M$ such that the following hold:
	\begin{enumerate}
		\item[(i)] For $\nu$-almost every $\gamma\in \Lambda $, the measure $\mu_{\gamma}$ is a $(\Gamma,\omega)$-Jacobi needle along $\gamma$.
		\item[(ii)] (``disintegration of measure'') For any measurable set $S \susbeteq M$,
		\begin{equation}\label{horodisinteqn}
		\int_S\omega = \int_{\Lambda} \mu_{\gamma}(S) d \nu(\gamma).
		\end{equation}
		
		\item[(iii)] (``mass balance'') For $\nu$-almost any $\gamma \in \Lambda$,
		\begin{equation}\label{MBhorodecomp}	\int_{M} \rho_1 d \mu_{\gamma} = \int_{M} \rho_2 d \mu_{\gamma}, \end{equation}
		and moreover
		\begin{equation}\label{MBhoroendsdecomp}	\int_{M} \rho_1 d \mu_{\gamma^+} \le \int_{M} \rho_2 d \mu_{\gamma^+} \end{equation}
		whenever $\gamma^+$ is a positive end of $\gamma$. Here a curve $\gamma^+$ is said to be a positive end of $\gamma$ if it is a restriction of $\gamma$ to a subinterval with the same upper endpoint, and the measure $\mu_{\gamma^+}$ is the restriction of $\mu_\gamma$ to the image of $\gamma^+$.
	\end{enumerate}
\end{theorem}

\begin{proof}
	If $M$ is a geodesically-convex Finsler manifold and $\Gamma$ is its  geodesic spray, then the conclusion follows directly from \cite[Theorem 4.7]{AK}, except that the notion of a Jacobi needle is not discussed there, but we shall deal with this point below. 

	\medskip
	By assumption, there exists a geodesically convex Finsler metric $\cF$ on $M$ such that $\Gamma$ is projectively equivalent to the geodesic spray of $\cF$. Since the geodesics of $\cF$ coincide with $\Gamma$-geodesics as oriented curves, we immediately obtain conclusions (ii),(iii) of Theorem \ref{needledecomp} for our spray $\Gamma$. It remains to prove conclusion (i), namely, that $\mu_\gamma$ is a $(\Gamma,\omega)$-Jacobi needle for $\nu$-almost every $\gamma$. To this end we recall some facts and notations from the proof of \cite[Theorem 4.7]{AK}.

	\medskip
	For $\nu$-almost every measure $\mu_\gamma$, either $\mu_\gamma$ is a Dirac mass, in which case it is trivially a $(\Gamma,\omega)$-Jacobi needle, or else it takes the following form. There exist
	\begin{enumerate}
		\item a Borel set $B\subseteq\RR^2$ of the form
		\begin{equation}\label{lineclusterformk}
					B = \left \{ (y,t) \in \RR^2 \, ; \, y \in Y, \ a_{y} < t < b_{y} \right \},
		\end{equation}
		where $Y \subseteq \RR$ is a Borel set and $a_y,b_y$ are measurable in $y$ with $a_y<b_y$, and
		\item a locally-Lipschitz function $F : B \to M$ such that $F(y,\cdot)$ is a constant-speed geodesic of $\cF$ for almost every $y \in Y$,
	\end{enumerate}
	and $\mu_\gamma$ is given by
	$$\mu_\gamma : = \gamma_\#\left(c \cdot |\det dF(y_0,t)|dt\right)$$
	for some $y_0\in Y$ and some $c > 0$. In particular $F$ is differentiable in $(y_0,t)$ for all $t \in (a_{y_0},b_{y_0})$. The determinant here is defined by $\det dF = (F^*\omega)(\partial_y,\partial_t)$. It is also proved in Lemma 4.9 in \cite{AK} that the function $t\mapsto\det dF(y_0,t)$ does not change sign. By precomposing $F$ with a map of the form $(y,t) \mapsto (\ell(y),t)$ where $\ell$ is affine, we may assume that $y_0 = 0$, $\det dF(y_0,t) > 0$ for all $t$, and $c = 1$. It is also not hard to replace $F$ (which is only known to be locally-Lipschitz) by a smooth variation through geodesics of $\cF$ which has the same Jacobian determinant at $s=0$. We then immediately see that $\mu_\gamma$ is a $\Gamma_\cF$-Jacobi needle, where $\Gamma_\cF$ is the geodesic spray of $\cF$. Since $\Gamma$ is projectively equivalent to $\Gamma_\cF$, it follows from Lemma \ref{jacprojlemma} that $\mu_\gamma$ is a $(\Gamma,\omega)$-Jacobi needle.
\end{proof}

With the needle decomposition theorem at hand, we are now ready to prove Theorem \ref{sprayBMthm}.

\begin{proof}[Proof of Theorem \ref{sprayBMthm}]
The proof is practically the same as the proof of \cite[Theorem 1.1]{AK}. The idea is to use Theorem \ref{needledecomp} to decompose the measure $\mu$ into a family of $(\Gamma,\omega)$-Jacobi needles, prove Brunn-Minkowski on each needle (Lemma \ref{lem_924}) using the assumption of nonpositive curvature of the spray, and then integrate the one-dimensional inequalities to obtain \eqref{sprayBM}. 

\medskip
Let $A,B \subseteq M$ be nonempty, Borel measurable sets and let $0<\lambda<1$. The set $\cM_\Gamma(A,B;\lambda)$ is Lebesgue measurable. Indeed, the set
$$\{(x,y,m) \mid x,y\in M, \, m\in \cM_\Gamma(\{x\},\{y\};\lambda)\}$$
is a closed subset of $M^2 \times M$, and we have 
$$\cM_\Gamma(A,B;\lambda) = \pi_2\left(\pi_1^{-1}(A\times B)\right),$$
where $\pi_1: M^2\times M \to M^2$ and $\pi_2: M^2\times M \to M$ are the projections, which are Borel-measurable.

\medskip
We may assume, by a standard approximation argument, that both $A$ and $B$ are compact, and in particular,
 $\mu(A)$ and $\mu(B)$ are finite. Apply Theorem \ref{needledecomp} with
\begin{equation}\label{rho12def} \rho_1 = \frac{\chi_A}{\mu(A)} \qquad \text{ and } \qquad \rho_2 = \frac{\chi_B}{\mu(B)}\end{equation}
to obtain measures $\{\mu_\gamma\}_{\gamma \in \Lambda}$ and $\nu$ with the properties (i)-(iii) in Theorem \ref{needledecomp}. Here $\chi_A$ is the indicator
function of the set $A$.
By \eqref{MBhorodecomp} and \eqref{rho12def}, for $\nu$-almost any $\gamma \in \Lambda$, if $0 < \mu_\gamma(A) <\infty$ then
\begin{equation}
\frac{\mu_{\gamma}(A)}{\mu(A)} = \frac{\mu_{\gamma}(B)}{\mu(B)}. \label{eq_1247}
\end{equation}
Since $A$ has finite measure, by \eqref{horodisinteqn} we know that $\mu_\gamma(A)  < \infty$ for $\nu$-almost any $\gamma \in \Lambda$. 

\begin{lemma}\label{lem_924} For $\nu$-almost any $\gamma \in \Lambda$, if $0 < \mu_\gamma(A) < \infty$ then
\begin{equation}\label{muconcavity}
	\mu_\gamma(\cM_\Gamma(A,B;\lambda))^{1/2} \ge (1-\lambda)\, \mu_\gamma(A)^{1/2} + \lambda\, \mu_\gamma(B)^{1/2}.
\end{equation}
\end{lemma}

\begin{proof}
	For $\nu$-almost any $\gamma \in \Lambda$, the measure $\mu_\gamma$ is a $(\Gamma,\omega)$-Jacobi needle along $\gamma$. From the definition of a $(\Gamma,\omega)$-Jacobi needle and the assumption that $\Gamma$ is nonnegatively curved with respect to $\omega$, it follows that there exists an interval $I\subseteq \RR$ and a measure $m$ on $I$ with a concave density, such that $\mu_\gamma = \gamma_\#m$. From the definition of $\cM_\Gamma(A,B;\lambda)$, it suffices to prove that
	$$m(\cM(\tilde A,\tilde B;\lambda))^{1/2} \ge (1-\lambda)\cdot m(\tilde A)^{1/2} + \lambda\cdot m(\tilde B)^{1/2},$$
	where $\tilde A := \gamma^{-1}(A)$, $\tilde B := \gamma^{-1}(B)$ and
	$$\cM(\tilde A,\tilde B;\lambda) : = \{(1-\lambda)a + \lambda b \mid a \in \tilde A, b \in \tilde B, a \le b\}.$$
	From this point the proof is identical to the proof in \cite[Section 3.2]{AK}, using conclusion (iii) of Theorem \ref{needledecomp} together with a variant of the one-dimensional Borell-Brascamp-Lieb inequality.
\end{proof}

We now integrate inequality \eqref{muconcavity} over $\gamma$ to obtain inequality \eqref{sprayBM}. By  \eqref{horodisinteqn}, \eqref{muconcavity} and \eqref{eq_1247},

\begin{align}
\begin{split}\label{finalcomputation}
	\mu(\cM_\Gamma(A,B;\lambda)) & \stackrel{\text{\eqref{horodisinteqn}}}{=} \int_{\Lambda}\mu_\gamma\left(\cM_\Gamma(A,B;\lambda)\right)d\nu(\gamma)\\
	& \stackrel{\text{\eqref{muconcavity}}}{\ge} \int_{\Lambda}\left((1-\lambda)\,\mu_\gamma(A)^{1/2} + \lambda\,\mu_\gamma(B)^{1/2}\right)^2d\nu(\gamma)\\
	& = \int_\Lambda\mu_\gamma(A)\left((1-\lambda) + \lambda\,\left(\frac{\mu_\gamma(B)}{\mu_\gamma(A)}\right)^{1/2}\right)^2d\nu(\gamma)\\
	& \stackrel{\text{\eqref{eq_1247}}}{=} \int_{\Lambda}\mu_\gamma(A)\left((1-\lambda) + \lambda\,\left(\frac{\mu(B)}{\mu(A)}\right)^{1/2}\right)^2d\nu(\gamma)\\
	& \stackrel{\text{\eqref{horodisinteqn}}}{=} \mu(A)\left((1-\lambda) + \lambda\,\left(\frac{\mu(B)}{\mu(A)}\right)^{1/2}\right)^2\\
	& = \left((1-\lambda) \, \mu(A)^{1/2} + \lambda\,\mu(B)^{1/2}\right)^2,
\end{split}
\end{align}
and \eqref{sprayBM} is proved.
\end{proof}

One can use the same proof to show that for a two-dimensional, pFm weighted spray space, and for $N\in [2,\infty)$ the $CD(0,N)$ condition from Remark \ref{CDrmk} implies the Brunn-Minkowski inequality with exponent $1/N$.

\begin{theorem}\label{sprayBMCDthm}
	Let $(M,\Gamma,\omega)$ be a two-dimensional weighted spray space satisfying the $CD(0,N)$ condition from Remark \ref{CDrmk} for some $N \in [2,\infty)$. Suppose that $\Gamma$ is projectively Finsler-metrizable. Then for every nonempty Borel subsets $A,B \subseteq M$ of positive measure and every $0<\lambda<1$,
	\begin{equation}\label{sprayBMCD}\tag{BM}
			\mu(\cM_\Gamma(A,B;\lambda))^{1/N} \ge (1-\lambda)\cdot\mu(A)^{1/N} + \lambda\cdot\mu(B)^{1/N}
	\end{equation}
	where $\mu$ is the Borel measure induced by the volume form $\omega$.
\end{theorem}

We now prove a converse to Theorem \ref{sprayBMthm}. The end of the ensuing proof is similar to the proof that if a $1/2$-concave measure on the real line has a continuous density then this density is concave; this is an instance of a more general theorem about $s$-concave measures in $\RR^n$, see Borell \cite{Bo}. See also \cite[Theorem 3.17]{DJ}. The only difference, which is completely immaterial to the proof, is that in our case we only know inequality \eqref{Jintinequ} when $x_0 \le x_1$.

\begin{proposition}\label{conversecor}
	Let $(M,\Gamma,\omega)$ be a two-dimensional simple weighted spray space. Assume that \eqref{sprayBM} holds for every Borel nonempty $A,B \subseteq M$ and every $0 < \lambda < 1$. Then $(M,\Gamma,\omega)$ is nonnegatively curved.
\end{proposition}
\begin{proof}
	
	Let $\gamma$ be a $\Gamma$-geodesic and let $S$ be a transversal $\Gamma$-Jacobi field along $\gamma$. Let $F(s,t):[-\delta_0,\delta_0]\times I \to M$ be a variation through curves of $\Gamma$ which induces the $\Gamma$-Jacobi field $S$ along $\gamma$. By transversality of $S$, we may take $F$ to be a diffeomorphism. Denote
    $$S = F_*(\partial/\partial s), \qquad T = F_*(\partial/\partial t)$$
    (which, by our choice of $F$, is consistent with the previous definition of $S$ at $s=0$). By the definition of a nonpositively curved weighted spray, we should prove that the function
    $$J(t) := \omega(\dot\gamma(t),S(t)) = \omega(T,S)\vert_{s=0}$$
    is concave. Let $[x_0,x_0 + \ell_0]$ and $[x_1,x_1 + \ell_1]$ be subintervals of $I$ with $x_0 + \ell_0 \le x_1$ and let $0<\lambda<1$. Choose $0 < \delta < \delta_0$ and set
    $$A : = F([0,\delta]\times [x_0,x_0 + \ell_0]), \qquad B : = F([0,\delta]\times [x_1,x_1 + \ell_1]).$$
    Since $F^*\omega = \omega(T,S) \, dt\wedge ds$, we have
    \begin{equation}\label{VolAVolB}
    \mu(A) = \delta \, \int_{x_0}^{x_0 + \ell_0}J(t)dt + o(\delta), \qquad \text{ and } \qquad \mu(B) = \delta \, \int_{x_1}^{x_1 + \ell_1}J(t)dt + o(\delta).
    \end{equation}
    Moreover, since $\Gamma$ is simple, there is a unique $\Gamma$-geodesic joining any two points in $F([0,\delta]\times I)$, which depends smoothly on its endpoints. Uniqueness implies that we can choose $\delta$ small enough that $\cM_\Gamma(A,B;\lambda)\subseteq F([0,\delta] \times I)$, because a $\Gamma$-geodesic joining $A$ to $B$ will not intersect $\partial F([0,\delta]\times I)$ twice. The smooth dependence of a $\Gamma$-geodesic on its endpoints implies that, for every $\eps > 0 $, we may choose $\delta$ small enough that
    $$\cM_\Gamma(A,B;\lambda) \subseteq F([0,\delta] \times [x_\lambda - \eps, x_\lambda + \ell_\lambda + \eps])$$
    where
    $$x_\lambda = (1-\lambda)x_0 + \lambda x_1 \qquad \text{ and } \qquad \ell_\lambda = (1-\lambda)\ell_0 + \lambda\ell_1.$$
    Therefore
    \begin{align}\label{VolABlambda}
    \begin{split}
    \mu(\cM_\Gamma(A,B;\lambda)) & \le \mu(F\left([0,\delta]\times [x_\lambda - \eps, x_\lambda + \ell_\lambda + \eps]\right)\\
    & = \delta \int_{x_\lambda - \eps}^{x_\lambda + \ell_\lambda + \eps}J(t)dt + o(\delta).
    \end{split}
    \end{align}
    Combining \eqref{VolAVolB}, \eqref{VolABlambda} and \eqref{sprayBM}, and taking $\delta,\eps \to 0$, we get
    \begin{equation}\label{Jintinequ}\left(\int_{x_\lambda}^{x_\lambda + \ell_\lambda}J(t)dt\right)^{1/2} \ge (1-\lambda) \cdot \left(\int_{x_0}^{x_0 + \ell_0}J(t)dt\right)^{1/2} + \lambda \cdot \left(\int_{x_1}^{x_1 + \ell_1}J(t)dt\right)^{1/2}\end{equation}
    for all $x_0,x_1$ and $\ell_0,\ell_1$ as above. By considering the first-order Taylor approximations of the integrals in \eqref{Jintinequ}, we see that if $\ell_0,\ell_1$ are sufficiently small then
    $$\left(\ell_\lambda J(x_\lambda)\right)^{1/2} \ge (1-\lambda)\left(\ell_0 J(x_0)\right)^{1/2} + \lambda\left(\ell_1 J(x_1)\right)^{1/2}.$$
    But since both sides are homogeneous in the $\ell_j$, the above inequality holds for all $\ell_0,\ell_1 \ge 0$. Set $\ell_j = J(x_j)$. Then
    $$\Big(\big((1-\lambda)J(x_0) + \lambda J(x_1)\big)J(x_\lambda)\Big)^{1/2} \ge (1-\lambda)J(x_0) + \lambda J(x_1),$$
    whence $J(x_\lambda) \ge (1-\lambda)J(x_0) + \lambda J(x_1).$ Since $x_0,x_1$ are arbitrary, we conclude that $J$ is concave.
\end{proof}

\begin{corollary}\label{simplemetricBMcor}
	Let $(M,g)$ be a Riemannian surface, let $\Gamma = E_1 + {k} V$ be a simple, proper metric spray on $M$ and let $\omega = e^{-\vphi}\omega_g$ be a volume form on $M$. The following are equivalent:
	\begin{enumerate}
		\item The spray $\Gamma$ is magnetic with respect to the metric $e^{-4\vphi}g$, and 
		\begin{align}\label{nncRiemeq2}K + {k}^2  +\Gamma(d\vphi) - (d\vphi)^2 - E_2{k} \ge 0.\end{align}
		In particular, if $\omega = \omega_g$ then $\Gamma$ is magnetic with repect to $g$ and
		\begin{equation}\label{nncRiemeqmag}K + \kappa^2 - |\nabla\kappa| \ge 0,\end{equation}
		where $\kappa$ is the geodesic curvature function of $\Gamma$.
		\item Inequality \eqref{sprayBM} holds for every Borel, nonempty subsets $A,B \subseteq M$ and every $0 < \lambda < 1$.
	\end{enumerate}
\end{corollary}

\begin{proof}
	Assume that $1$ holds. By Corollary \ref{nncriemcor}, the weighted spray space $(M,\Gamma,\omega)$ is nonnegatively curved. Let $A,B$ be Borel, nonempty subsets of $M$. By an approximation argument, we may assume that both are compact. Since $\Gamma$ is proper, there exists a $\Gamma$-convex open set $U$ containing $A$ and $B$, and by Proposition \ref{metrizabilityppn}, the restriction of $\Gamma$ to $U$ is Finsler-metrizable; we may assume without loss of generality that $M = U$. If $A,B$ have positive measure, inequality \eqref{sprayBM} then follows from Theorem \ref{sprayBMthm}. 

	\medskip
	Suppose that one of the sets has zero measure. In the case where $\mu(A) = \mu(B) = 0$, inequality (\ref{sprayBM}) holds trivially.
	Suppose that $\mu(A) = 0$ but $\mu(B) > 0$. Since $A$ is non-empty, we may pick a point $a \in A$. For $t > 0$ define a map $H_t : M \to M$ by
	\begin{equation}\label{Htdef}
		H_t(x) = \exp_a\left(t \cdot \exp_a^{-1}(x)\right),
	\end{equation}
	where $\exp := \exp^\Gamma$, and observe that 
	$$\cM_\Gamma(A,B;\lambda) \supseteq \cM_\Gamma(\{a\},B;\lambda) = H_\lambda(B).$$
	Thus, since $A$ has measure zero, in order to prove \eqref{sprayBM} it suffices to show that for $0<\lambda<1$,
	$$\mu(H_{\lambda}(B)) \ge \lambda^2 \cdot \mu(B).$$
	Let $\tilde\omega = \exp_a^*\omega$ and let $\tilde\mu$ be the Borel measure on $T_aM$ induced by the volume form $\tilde\omega$. Making a change of variables via the map $\exp_a$, which is a $C^1$ diffeomorphism since $\Gamma$ is simple, and using \eqref{Htdef}, the last inequality becomes
	$$\tilde\mu(\lambda\tilde B) \ge \lambda^2 \cdot \tilde\mu(\tilde B),$$
	where $\tilde B := \exp_a^{-1}(B) \subseteq T_aM$. Introduce polar coordinates $(r,\theta)$ on $T_aM$, and write $\tilde \omega = J(r,\theta)\,dr\wedge d\theta$ for a smooth function $J$ on $T_aM\setminus \{0\}$. By Fubini's theorem, it suffices to prove that for every $\theta$ and every measurable subset $S \subseteq \RR$,
	$$\int_{\lambda S} J(r,\theta)dr \ge \lambda^2 \, \int_S J(r,\theta)dr.$$

	If we set $T : = (\exp_a)_*\partial_r$ and $S : = (\exp_a)_*\partial_\theta$ then $J = \omega(T,S)$, and the vector field $S$ is a $\Gamma$-Jacobi field along each $\Gamma$-geodesic of the form $r \mapsto \exp_a((r\cos\theta,r\sin\theta))$ whose tangent is $T$. Thus the function $r \mapsto J(r,\theta)$ is concave for every $\theta$ since $(M,\Gamma,\omega)$ is nonnegatively curved. Moreover $J \to 0$ as $r \to 0$ since $d\exp_a$ is the identity at the origin. We therefore have
	$$\int_{\lambda S}J(r,\theta)dr = \lambda \int_S J(\lambda r,\theta)dr = \lambda\int_S J((1-\lambda)\cdot 0 + \lambda \cdot r,\theta)dr \ge \lambda^2\int_S J(r,\theta)dr,$$
	as desired. The case $\mu(B) = 0$ and $\mu(A)>0$ follows by reversing the spray $\Gamma$ (i.e. replacing it with the spray whose geodesics are $\Gamma$-geodesics traversed backwards) and applying the previous case. This finishes the proof of the implication 1 $\implies$ 2. 
	
	\medskip
	The implication 2 $\implies$ 1 follows from Proposition \ref{conversecor} and Corollary \ref{nncriemcor}.
\end{proof}

\section{Examples}\label{examplesec}
In this section we give some examples of two-dimensional weighted spray spaces satisfying \eqref{sprayBM}.

\medskip
Consider the case of constant curvature $K\equiv \mathrm{const}$, and $\omega = \omega_g$. By Corollary \ref{simplemetricBMcor}, for a simple, proper metric spray, the Brunn-Minkowski inequality \eqref{sprayBM} holds if and only if the spray is magnetic with respect to the metric of constant curvature $g_K$ and the geodesic curvature function $\kappa$ satisfies \eqref{nncRiemeqmag}. Define 
    \begin{align*}
    \cot_Kx : = 
    \begin{cases}
        \sqrt{K}\cot(\sqrt{K}x) & K > 0,\\
        1/x & K = 0,\\
        \sqrt{-K}\coth(\sqrt{-K}x) & K < 0.
    \end{cases}
    \end{align*}
    Then \eqref{nncRiemeqmag} holds if and only if either $\kappa = \cot_K(f)$, where $f$ is smooth, 1-Lipschitz function on $M$, or $K \le 0$ and $\kappa \equiv \sqrt{-K}$. 
    \begin{example}[Horocycles]\normalfont
	    On the hyperbolic plane we can take $\kappa \equiv 1$. The resulting spray has as its geodesics constant-speed horocycles. Inequality \eqref{sprayBM} for this spray was proved in \cite{AK}. Since this spray is simple, metric and proper, \eqref{sprayBM} for horocycles follows from Corollary \ref{simplemetricBMcor}. 
	\end{example}
	\begin{example}[Norwich Spirals]\normalfont
	    Take $M = \RR^2\setminus\{0\}$ with the flat metric $g_0 = dx^2 + dy^2$ and $\kappa = 1/r$ where $r = \sqrt{x^2 + y^2}$. The geodesics of the corresponding spray are either circles centered at the origin, or so-called Norwich spirals \cite{Zw} which are curves of the form
	    \begin{equation}\label{norwich}\gamma_{a,b}(t) =  a(t^2 + 1)e^{i(t - 2\arctan t + b)}, \qquad t \in \RR\end{equation}
	    for $a > 0$ and $b \in (-\pi,\pi]$, reparametrized to have constant speed. This spray is not simple (its geodesics self-intersect), but it is projectively Finsler-metrizable by the Randers metric
	    $$\cF = \sqrt{dx^2 + dy^2} - rd\theta.$$
	    Indeed, $d(rd\theta) = dr\wedge d\theta = (1/r)(r\,dr\wedge d\theta) = (1/r)\,\omega_{g_0}$. The parametrization \eqref{norwich} is in fact proportional to arclength with respect to $\cF$, hence completeness (and in particular geodesic convexity) of the metric follows from completeness of the flow of the spray and the Hopf-Rinow theorem. Thus by Theorem \ref{sprayBMthm}, this weighted spray space satisfies \eqref{sprayBM}. Analogues of the Norwich spirals exist on the (punctured) sphere and hyperbolic plane and also satisfy \eqref{sprayBM}.
	\end{example}
	\begin{example}[Seiffert Spirals]\normalfont
	    Take $M = S^2 \subseteq \RR^3$ with the round metric $g_1$. 
	   	Consider the magnetic spray on $S^2$ whose geodesic curvature function is $\kappa = z$, where $z$ is the third coordinate in $\RR^3$. The geodesics of this spray which pass through the poles are known as Seiffert spirals \cite{Er}. If $\phi$ denotes spherical distance from the north pole, then $z = \cos(\phi) = \cot(f)$ where $f = \arccot(\cos(\phi))$ is a smooth, 1-Lipschitz function on $S^2$. Thus by the discussion above, \eqref{nncRiemeqmag} is satisfied. If $\phi,\theta$ are spherical coordinates, then the 1-form $\eta = (\sin^2\phi/2)d\theta$ satisfies $d\eta = \cos\phi \, \sin\phi \, d\phi \wedge d\theta = \cos\phi \, \omega_{g_1}$, hence the geodesic spray of the Randers metric $\sqrt{g_1} - \eta$ is projectively equivalent to this spray. This metric is geodesically convex by compactness and the Hopf-Rinow theorem. It thus follows from Theorem \ref{sprayBMthm} and Corollary \ref{nncriemcor} that this spray satisfies \eqref{sprayBM}. 
   \end{example}
	\begin{example}[Circular arcs]\normalfont\label{circex}
		Let $0 < r \le R$, let $D \subseteq \RR^2$ be an open disc of radius $r$ and consider the spray on $D$ whose geodesics are arcs of circles of radius $R$, parametrized proportionally to Euclidean arclength. This spray is simple, metric and magnetic (with respect to the Euclidean metric) and inequality \eqref{nncRiemeqmag} holds (with $K\equiv 0$ and $\kappa \equiv 1/R$). Thus this spray satisfies \eqref{sprayBM} by Corollary \ref{simplemetricBMcor}.
	\end{example}
	\begin{example}[Perturbation]\normalfont
		Under the assumptions and notation of Corollary \ref{simplemetricBMcor}, if inequality \eqref{nncRiemeq2} is strict, then a $C^1$-small perturbation of $k$ and/or a $C^2$-small perturbation of $\vphi$ will preserve inequality \eqref{sprayBM}, as long as the perturbed spray is still simple. For example, a simple metric spray on a spherical cap which is $C^1$-close to the geodesic spray will satisfy \eqref{sprayBM} with respect to the standard area measure, and the circular spray from Example \ref{circex} will satisfy \eqref{sprayBM} with respect to a $C^2$-small density on $D$.
	\end{example}

\end{document}